\documentclass[pdftex,12pt,a4paper]{article}

\usepackage[utf8]{inputenc}
\usepackage{amsmath}
\usepackage{amssymb}
\usepackage{amsthm}
\usepackage{tabularx}
\usepackage{xcolor}
\usepackage{tikz}
\usepackage{mathtools}
\usepackage{hyperref}
\usepackage[all]{xy}

\newtheorem{theorem}{Theorem}[section]
\newtheorem{lemma}[theorem]{Lemma}

\newtheorem{remark}[theorem]{Remark}
\newtheorem{definition}[theorem]{Definition}
\newtheorem{example}[theorem]{Example}

\newcommand{\Z}{\mathbb{Z}}
\newcommand{\Q}{\mathbb{Q}}
\newcommand{\C}{\mathbb{C}}
\newcommand{\F}[1]{\mathbb{F}_{#1}}

\title{A survey on recursive towers and Ihara's constant}
\date{}

\author{Peter Beelen}

\begin{document}
\maketitle

\begin{abstract}
Since Serre gave his famous Harvard lectures in 1985 on various aspects of the theory of algebraic curves defined over a finite field, there have been many developments. In this survey article, an overview will be given on the developments concerning the quantity $A(q)$, known as Ihara's constant. The main focus will be on explicit techniques and in particular recursively defined towers of function fields over a finite field, which have given good lower bounds for Ihara's constant in the past.
\end{abstract}

\section{Introduction}

Let $\F{q}$ denote the finite field with $q$ elements. Of course $q=p^n$ for some prime number $p$ and some positive integer $n$. Given such field, one may consider absolutely irreducible, projectively closed, nonsingular algebraic curves $\mathcal C$ defined over $\F{q}$. Given such an algebraic curve $\mathcal C$, we will denote by $N_1(\mathcal C)$, the number of its $\F{q}$-rational points. The well known Hasse--Weil bound states that $N_1(\mathcal C) \le q+1+2\sqrt{q} g(\mathcal C).$ For intrinsic reasons, as well as from motivations from coding theory, there has been a continuing interest in finding curves defined over $\F{q}$ having a particular genus $g$, with as many $\F{q}$-rational points as possible.
In other words, one has tried to determine the following quantity:
$$N_q(g)=\max\{N_1(\mathcal{C}) \mid \mathcal C \text{ a curve of genus $g$ over $\F{q}$}\}.$$
It was noticed relatively early on that the Hasse--Weil bound cannot always be achieved. In the first place, if $q$ is not a square, one has the trivial improvement $N_1(\mathcal C) \le q+1+\lfloor 2\sqrt{q} g(\mathcal C) \rfloor,$ but a much more significant refinement was given by Serre, \cite[Theorem 2.2.1]{Serre85}: $N_1(\mathcal C) \le q+1+\lceil 2\sqrt{q}\rceil g(\mathcal C).$ In the second place, Ihara noticed that even in case $q$ is a square, the Hasse--Weil bound cannot be attained for large genera. More precisely, in \cite{Ihara} he improved the Hasse--Weil bound for curves over $\F{q}$ and genus larger than $\sqrt{q}(\sqrt{q}-1)/2$. In the same paper, Ihara introduced what is now known as Ihara's constant:
\begin{definition}\label{def:Aq}
Let $\F{q}$ be a finite field with $q$ elements. Then define
$$A(q)=\limsup_{g \to \infty} N_q(g)/g.$$
\end{definition}
In \cite{Ihara} an important first step was taken in the determination of $A(q)$: if $q=p^{2m}$, then $A(q) \ge p^m-1$. Ihara used families of Shimura curves to obtain his result and he already contemplated the existence of such families of curves with many points in \cite{Ihara66}. In \cite{TVZ} families of classical modular curves were used to prove that $A(p^2) \ge p-1$ and variations were indicated proving $A(q) \ge \sqrt{q}-1$ for square values of $q$, thus independently obtaining the same lower bound as in \cite{Ihara}. Combined with the Drinfeld--Vladut upper bound $A(q) \le \sqrt{q}-1$, see \cite{DV}, this means that $A(q)=\sqrt{q}-1$ whenever $q$ is a square. For all other values of $q$, the precise value of Ihara's constant is not known to this date. The Drinfeld--Vladut bound is the best upper bound for $A(q)$ currently known. As for lower bounds, at the time of Serre's lecures these were scarce:
First of all, Serre showed in his Harvard lectures that there exists a constant $c>0$, not depending on $q$, such that $A(q) \ge c \log(q)$ for all $q$, see \cite[Theorem 5.10.1]{Serre85}. We will refer to this as Serre's logarithmic bound. In particular it implies that $A(q)>0$ for all $q$. In the second place Zink, \cite{Zink}, showed for every prime number $p$ using degeneration of Shimura surfaces, that $A(p^3) \ge 2(p^2-1)/(p+2)$, a result we will refer to as Zink's bound.

Since then, various developments have taken place giving rise to better lower bounds on $A(q)$. In the remaining part of this article, we will give an overview of the used techniques to obtain these lower bounds. As such very few new results will appear in this article and an effort has been made to provide extensive references. Apart from the geometric language of curves defined over a finite field, we will also regularly use the language of function fields with a finite constant field. We will indicate references to specific results throughout the article, but as general references for curves defined over a finite field or equivalently function fields with finite constant field, the reader is encouraged to consult \cite{NX01,Serre85,Stich09,TVN1,TVN2}.

\section{Early methods}

In this section, we give an overview of early methods, meaning methods used before and at the time of Serre's Harvard lecture, to find lower bounds for $A(q)$. Here `early' should not be read as `superseded', but simply as an historical note as to when these methods were first used. In particular, several of the results described in this section were published after 1985.

\subsection{Bounds using class field theory}\label{subsec:classfield}

Serre initiated the use of class field towers in the study of Ihara's constant. In particular Serre's aforementioned result $A(q) \ge c \log(q)$ for all $q$, was shown using class field theory. In \cite[Theorem 5.2.9]{NX01} it is shown that one can take $c=\frac{1}{96 \log(2)}$. Serre already demonstrated in his Harvard lectures, that for $q=2$, class field theory can be used to show that $A(2) \ge 2/9$ \cite[Theorem 5.11.1]{Serre85}, a result that was also obtained using a different construction in \cite{Schoof92}.

Following this there has been a significant activity trying to improve this lower bound as well as to obtain lower bounds for $A(q)$ for small values of $q$. In particular for $q=2$ and $q=3$ increasingly better lower bounds have been found through the years, using infinite class field towers. The following table gives an overview:
$$
\begin{array}{ll|ll}
A(2) \ge 2/9\approx 0.222222...      & \cite{Serre85,Schoof92} & A(3) \ge 62/163 \approx 0.380368... & \cite{NX98}\\
A(2) \ge 81/317 \approx 0.255520...  & \cite{NX98}             & A(3) \ge 8/17 \approx 0.470588... & \cite{Tem01,AM02}\\
A(2) \ge 97/376 \approx 0.257978...  & \cite{XY07}             & A(3) \ge 12/25 = 0.48 & \cite{HM01}\\
A(2) \ge 39/129 \approx 0.302325...  & \cite{Kuhnt}            & A(3) \ge 0.492876... & \cite{DM13}\\
A(2) \ge 0.316999...                 & \cite{DM13}             &            &
\end{array}
$$
The lower bounds for $A(2)$ and $A(3)$ found in \cite{DM13} are currently the best known. Like the other displayed bounds, these bounds are rational numbers, but numerator and denominator have been left out, since they are rather large. Also for other small prime values of $q$, the same method has been used successfully. An overview for $q \in \{5,7,11\}$:
$$
\begin{array}{ll|ll}
A(5) \ge 2/3 \approx 0.666666...  & \cite{NX98}       & A(7) \ge 12/13\approx 0.9230769... & \cite{HS13}\\
A(5) \ge 8/11 \approx 0.727272... & \cite{Tem01,AM02} & A(11) \ge 12/11\approx 1.090909... & \cite{LM02}\\
A(7) \ge 9/10=0.9                 & \cite{LM02}       & A(11) \ge 8/7\approx 1.142857...   & \cite{HS13}\\
\end{array}
$$
There seems to be little doubt that in the future class field towers can be used to improve these type of results further in case $q$ is a prime. Also in case $q$ is not a prime, class field theory has been used to obtain lower bounds for $A(q)$. However, while in the prime case class field towers have so far given the best results, this is not true in the non-prime setting. There, other methods have produced the currently best known lower bounds for $A(q)$, which we will address later in this article. Nonetheless, also for the non-prime case, class field theory is an interesting approach. Let us consider some small non-prime values of $q$. First of all, if $q$ is a square, we have already mentioned that the value of $A(q)$ is known. Hence, the smallest non-prime value of $q$ for which $A(q)$ is not known is $q=8$. For $q=8$, Zink's bound implies that $A(8) \ge 3/2$, while the Drinfeld--Vladut bound states that $A(8) \le \sqrt{8}-1 \approx 1.828427...$. Zink's bound $A(8) \ge 3/2$ is currently the best known lower bound for $A(8)$ which can be also be obtained by techniques addressed later in this article. It is of course possible and sometimes conjectured that $A(8)=3/2$, but class field towers may give a way to disprove this conjecture.

\subsection{Explicit equations for modular curves}

One of the driving motivations for studying families of curves with many $\F{q}$-rational points, is that using Goppa's construction of error-correcting codes, such families can be used to find good families of such codes \cite{TVZ}. To understand such codes as well as possible, it is therefore of interest to describe the used curves as explicitly as possible. Since in \cite{TVZ} classical modular curves were used as a source of curves with many rational points, it is therefore natural to spend a bit on time on explicit models for such curves. The description here will be somewhat ad hoc and incomplete: for a more thorough description, see \cite[Chapter5]{TVN2} and the references therein.

For a positive integer $N$, denoted as the level, a bivariate polynomial $\Phi_N(X,Y) \in \Z[X,Y]$ can be constructed, which is called the modular polynomial of level $N$, having the following remarkable property: for each pair of $N$-isogenous elliptic curves $(E_0,E_1)$ defined over $\C$, that is to say, for each pair of elliptic curves for which there exists and cyclic isogeny $\lambda: E_0 \to E_1$ of order $N$, it holds that $\Phi_N(j_0,j_1)=0$, where $j_0:=j(E_0)$ and $j_1:=j(E_1)$ are the $j$-invariants of $E_0$ and $E_1$ respectively. Conversely, if $\Phi_N(j_0,j_1)=0$ for complex numbers $j_0$ and $j_1$, then there exists a pair of $N$-isogenous elliptic curves $(E_0,E_1)$ over $\C$ with $j$-invariants $j_0$ and $j_1$. The affine curve $Y_0(N)$ defined by the polynomial equation $\Phi_N(X,Y)=0$ may have singularities, but the modular polynomial can be used to describe its function field as $\Q(Y_0(N))=\Q(x_0,x_1)$, where $x_0$ is a transcendental element over $\Q$ and $x_1$ satisfies $\Phi_N(x_0,x_1)=0$. The projective closure of $Y_0(N)$ is denoted by $X_0(N)$ and of course the function field of $X_0(N)$ is isomorphic to that of $Y_0(N)$. To obtain a curve defined over a finite field, one may reduce the defining equation $\Phi_N(x_0,x_1)=0$ modulo a prime not dividing the level $N$ and in this way obtain a function field with constant field $\F{p}$. Extending the constant field to $\F{p^2}$ then results in a function field with many rational places. The underlying reason this function field has many rational places is that pairs of $N$-isogenous supersingular elliptic curves give rise to $\F{p^2}$-rational places. Note that if $\lambda: E_0 \to E_1$ is a cyclic isogeny of degree $N$, then the dual isogeny $\hat{\lambda}: E_1 \to E_0$ is cyclic of degree $N$ as well. This implies that $\Phi_N(X,Y)$ is a symmetric polynomial.

It is somewhat contrived to consider $N=1$. In this case, $X_0(1)$ is simply the projective line, and in this context usually denoted by $X(1)$. The simplest nontrivial example of a modular polynomial is therefore obtained for $N=2$:
\begin{multline}\Phi_2(X,Y)=X^3+Y^3-162000(X^2+Y^2)+1488(X^2Y+XY^2)-\\X^2Y^2+8748000000(X+Y)+40773375XY-157464000000000.\end{multline}
To describe a family of function fields with increasing genera and number of $\F{p^2}$-rational places, one can simply choose reduction modulo a suitable prime $p$ of the function fields of $(X_0(N_i))_{i}$ where $N_i \to \infty$ as $i \to \infty$. It is required that the prime of reduction does not divide any of the levels $N_i$. Such families achieve the Drinfeld--Vladut bound $A(p^2)=p-1$. 

While classical modular curves give rise to curves with many $\F{p^2}$-rational points, Drinfeld modular curves can be used to obtain curves with many $\F{q^2}$-rational points. Whereas classical modular curves $X_0(N)$ are determined by the level $N$, which is a a nonnegative integer, Drinfeld modular curves are in the simplest case determined by a polynomial $N(T) \in \F{q}[T]$. The role of the ring of integers $\Z$ is now played by the polynomial ring $\F{q}[T]$. In particular, the Drinfeld modular polynomial of level $N(T)$, which we will denote by $\Phi_{N(T)}(X,Y)$, has coefficients in $\F{q}[T]$. The role of elliptic curves is now played by rank two Drinfeld modules. Similar to the classical modular case, the function field of a Drinfeld modular curve $X_0(N(T))$ can be described as $\F{q}(T)(x_0,x_1)$, where $x_0$ is an element transcendental over $\F{q}(T)$, while $x_1$ satisfies $\Phi_{N(T)}(x_0,x_1)=0$. More generally, given a function field $F$ with finite constant field $\F{q}$ and a place $P$ of $F$, one can define a modular curve with level some non-trivial ideal in the ring $A \subset F$ of functions having no poles outside $P$. The case just described arises by choosing $F=\F{q}(T)$, the rational function field, and $P=P_\infty$, the place ``at infinity''. In this case $A=\F{q}[T]$ and a non-trivial ideal in $A$ can be described by its monic generator $N(T)$. See \cite{Gekeler86} for details on the general case. Similarly as in the classical case, a function field with constant field $\F{P(T)}:=\F{q}[T]/\langle P(T)\rangle$ is obtained from the function field of $X_0(N(T))$, by reducing its defining equation $\Phi_{N(T)}(x_0,x_1)=0$ modulo an irreducible polynomial $P(T) \in \F{q}[T]$, where it is required that $P(T)$ does not divide the level $N(T)$. Then over a quadratic extension of $\F{P(T)}$, which we will denote by $\F{P(T)}^{(2)}$, the resulting function field has many rational places.

Not many explicit examples of Drinfeld polynomials are known, except if $q$ is chosen to be small and $N(T)$ of low degree; see for example \cite{Schweizer,BBN14}. For general $q$, an explicit expression for $\Phi_T(X,Y)$, in some sense the easiest case, was determined in \cite{BB12}:
\begin{multline*}
{\Phi_T(X,Y)=(X+Y-j_0)^{q+1}-XY^q-X^qY+(XY)^q(T^{1-q}-1)+}\\
{XY(T^{q-1}-1)^{q^2}-T^{1-q}XY\sum_{\scriptscriptstyle i=0}^{\scriptscriptstyle \lfloor\frac{q-1}{2}\rfloor}C_i \cdot (X Y-T^q (X+Y-j_0))^{q-1-2i}(X Y T^{q^2+1})^{i},}
\end{multline*}
where $j_0=-T(T^{q-1}-1)^{q+1}$ and $C_i=\frac{1}{i+1}\binom{2i}{i}$ is the $i$-th Catalan number.
Similar as in the classical modular case, to describe a family of function fields with increasing genera and number of $\F{P(T)}^{(2)}$-rational places, one can simply choose reduction modulo an irreducible polynomial $P(T) \in \F{q}[T]$ of the function fields of $(X_0(N_i(T)))_{i}$ where $\deg N_i(T) \to \infty$ as $i \to \infty$. It is required that $P(T)$ does not divide any of the levels $N_i(T)$. Such families achieve the Drinfeld--Vladut bound $A(q^{2 \deg P(T)})=q^{\deg P(T)}-1$.

It is worth noting that in \cite{MV84}, an algorithm, polynomial in $q^{\deg N(T)}$, is given to find a defining equation for curves closely related to $X_0(N(T))$ in case $N(T)$ is irreducible. More precisely, the curves in question are $X_1(N(T))$, a curve covering $X_0(N(T))$ whose affine points correspond to isomorphism classes of a rank two Drinfeld module together with an $N(T)$-torsion point. Also these curves can be used as families of curves attaining the Drinfeld--Vladut upper bound for any square value of $q$.

\section{Recursively defined towers of function fields.}

In the mid 90's, another method to find lower bounds for $A(q)$ appeared. Explicit and conceptually simple constructions of families of function fields were discovered that lead to alternative proofs of Ihara's result that $A(q^2) \ge q-1$, but also gave lower bounds for $A(q)$ in case $q$ is not a square.

\subsection{The first two Garcia--Stichtenoth towers}

As we have seen in the previous section, the problem of finding explicit families of curves attaining the Drinfeld--Vladut upper bound can in principle be solved algorithmically. However, the approach described in the previous section does require one to restart the computations each time a new curve in the family needs to be computed. In 1995 a very different approach was followed by A.~Garcia and H.~Stichtenoth, leading to an explicit family of function fields over a finite field $\F{q^2}$ and many rational places, see \cite{GS95}. They defined the function field $E_{n}=\F{q^2}(x_0,\dots,x_n)$, where $x_0$ is a transcendental element over $\F{q^2}$ and
\begin{equation}\label{eq:INV95}
x_{i-1}^{q-1}x_{i}^q+x_{i}=x_{i-1}^q \quad \text{for $i=1,\dots,n$}.
\end{equation}
As $n \to \infty$ the ratio of number of places of degree one of $E_n$, denoted by $N_1(E_n)$, and the genus of $E_n$, denoted by $g(E_n)$, tends to $q-1$, achieving the Drinfeld--Vladut bound. At each step $n$, the defining equations of $E_n$ are known explicitly and obtained by adding one equation to those of $E_{n-1}$. The idea to use essentially the same equation recursively by increasing the indices of the variables in each recursive step, came from G.L.~Feng and T.R.N.~Rao, who were interested in algebraic curves because of their applications in coding theory. More specifically, they considered the equations
\begin{equation}\label{eq:FR}
x_{i-1}x_i^3+x_i=x_{i-1}^3
\end{equation}
over the finite field $\F{8}$. Though these equations have many solutions over $\F{8}$, it turns out that the genus of the resulting function fields grows too fast \cite[Section 4]{GS96}. R.~Pellikaan, interested in their construction, suggested to use the equations $x_{i-1}x_i^2+x_i=x_{i-1}^2$ over $\F{4}$ instead. This was the starting point of the equations studied in \cite{GS95}, giving a good example of different scientific areas inspiring each other. In particular the following definition came out of this:
\begin{definition}
A tower over $\F{q}$ is an infinite sequence $\mathcal{F} = (F_n)_{n \ge 0}$ of function fields $F_n/\F{q}$ with full constant field $\F{q}$, such that the following hold:
\begin{enumerate}
\item[(i)] $F_0 \subsetneq F_1 \subsetneq F_2 \subsetneq \cdots \subsetneq  F_n \subsetneq \cdots$;
\item[(ii)] for each $n\ge 1$, the extension $F_{n}/F_{n-1}$ is finite and separable;
\item[(iii)] the genera satisfy $g(F_n) \to \infty$ as $n \to \infty$.
\end{enumerate}
The limit of the tower is $\lambda(\mathcal F)=\lim_{n \to \infty} \frac{N_1(F_n)}{g(F_n)}.$
\end{definition}
The point of introducing towers is that clearly $A(q) \ge \lambda(\mathcal F)$ for any tower $\mathcal F$ over $\F{q}.$ If a tower $\mathcal F$ over $\F{q}$ has limit $0$, respectively $>0$, respectively $\sqrt{q}-1$, it is called \emph{bad}, respectively \emph{good}, respectively \emph{optimal}.
It is not hard to see that equation \eqref{eq:INV95} as well as equation \eqref{eq:FR} recursively define towers of function fields over $\F{q^2},$ which we will denote by $\mathcal E$ and $\mathcal{D}$. It is a lot less obvious how to compute the limit of these towers.

Two notions are convenient to have when studying towers in general, especially recursively defined towers. The first notion is that of the \emph{splitting locus} $S(\mathcal F)$ of a tower $\mathcal F$:
$$\{P \in \mathbb{P}_{F_0} \mid \text{ $\deg(P)=1$ and for all $n$, $P$ splits completely in $F_n/F_0$}\}.$$
Here $\mathbb{P}_{F}$ denotes the set of places of a function field $F$. The usefulness of this notion lies in the fact that it gives rise to the estimate $N_1(F_n) \ge |S(\mathcal F)|[F_n:F_0].$ Hence a tower with a non-empty splitting locus satisfies $\lim_{n \to \infty} N_1(F_n)/[F_n:F_0] >0.$ The towers $\mathcal D$ and $\mathcal E$ both have nonempty splitting locus. Let us for convenience for any constant field $\F{}$, identify rational places of the rational function field $\F{}(x_0)$ with elements from $\F{} \cup \{\infty\}$. Rewriting equation \eqref{eq:FR} as $(x_{i-1}^4x_i)^3/x_{i-1}^7+(x_{i-1}^4x_i)=x_{i-1}^7$, the simple fact that $\alpha^7=1$ for any $\alpha \in \F{8}\setminus \{0\}$, then implies that $S(\mathcal D)$ contains $\F{8}\setminus \{0\}$. Similarly, rewriting equation \eqref{eq:INV95} as $(x_{i-1}x_i)^q+(x_{i-1}x_i)=x_{i-1}^{q+1}$ shows that $S(\mathcal E)$ contains $\F{q^2}\setminus \{0\}$. In both cases actually equality holds. In particular $N_1(E_n) \ge (q^2-1)[E_n:E_0]=(q^2-1)q^n.$ In order to guarantee many rational places, it is enough to suppose that there exists $i\ge 0$ and a rational place $P$ of $F_i$ that splits completely in the extension $F_n/F_i$ for all $n>i$. Such a place is called a splitting place of the tower. One could generalize the definition of splitting locus along these lines as well, but in this article we will not need this.

The second convenient notion is that of the \emph{ramification locus} $V(\mathcal F)$ of a tower $\mathcal F$:
$$\{P \in \mathbb{P}_{F_0} \mid \text{$P$ ramifies in $F_n/F_0$ for some $n$}\}.$$
For example, we have $V(\mathcal D)=\{0,\infty\}$ and also $V(\mathcal E)=\{0,\infty\}.$ Note that ramification needs not occur already in the first step $F_1/F_0$ of the tower. Using the Riemann--Hurwitz genus formula, we immediately obtain that
$$2g(F_n)-2 = (2g(F_0)-2)[F_n:F_0]+\sum_{P \in V(\mathcal F)} \sum_{Q|P} d(Q|P)\deg(Q).$$
If the ramification locus $V(\mathcal F)$ is a finite set and all ramification is \emph{$b$-bounded}, that is to say if for all places $P$ of $F_0$ and for all $n$ and places $Q$ of $F_n$ lying above $P$ it holds that $d(Q|P) \le b e(Q|P)$ for some $b \in \mathbb{R}_{>0}$, we then may conclude that
\begin{eqnarray*}
2g(F_n)-2  & \le & (2g(F_0)-2)[F_n:F_0]+\sum_{P \in V(\mathcal F)} \sum_{Q|P} b e(Q|P)f(Q|P)\deg(P)\\
 & = & (2g(F_0)-2+b\sum_{P \in V(\mathcal F)} \deg(P))[F_n:F_0].
 \end{eqnarray*}
In particular, this implies that $\lim_{n \to \infty} g(F_n)/[F_n:F_0] < \infty.$ If all ramification in the tower is tame, i.e. the characteristic does not divide any of the ramification indices, it is well known by Dedekind's different formula that $d(Q|P)=e(Q|P)-1$, meaning that all ramification is $1$-bounded. This means that the genus can be controlled in tame towers, i.e. towers $\mathcal F$ for which for all $n$, all ramification in $F_n/F_0$ is tame. In wild, i.e. non-tame, towers Dedekind's different formula does not apply, which is why technical genus computations were carried out in \cite{GS95,GS96} to compute the genera of the towers defined by equations \eqref{eq:INV95} and \eqref{eq:FR}. As mentioned before, the tower $\mathcal D$ turns out to be bad, in particular the ramification is not bounded, but the tower $\mathcal E$ over $\F{q^2}$ is actually optimal. Though the notion of bounded ramification was not used at the time \cite{GS95} was written, with hindsight one can see that the key-insight from that article is that all ramification in tower $\mathcal E$ is $(q+2)$-bounded. Then combining all the above, we obtain that $$\lambda(\mathcal E) \ge \lim_{n \to \infty} \frac{(q^2-1)q^n}{1+\frac{-2+(q+2)2}{2} q^n}=q-1.$$ Since the Drinfeld--Vladut bound implies that $\lambda(E) \le q-1,$ we may conclude that $\mathcal E$ is a good, even optimal, tower over $\F{q^2}$.

It is very convenient to think of a recursively defined tower such as $\mathcal D$ and $\mathcal E$ as one of the sides of a pyramid of function fields, see Figure \ref{fig:pyramid}.
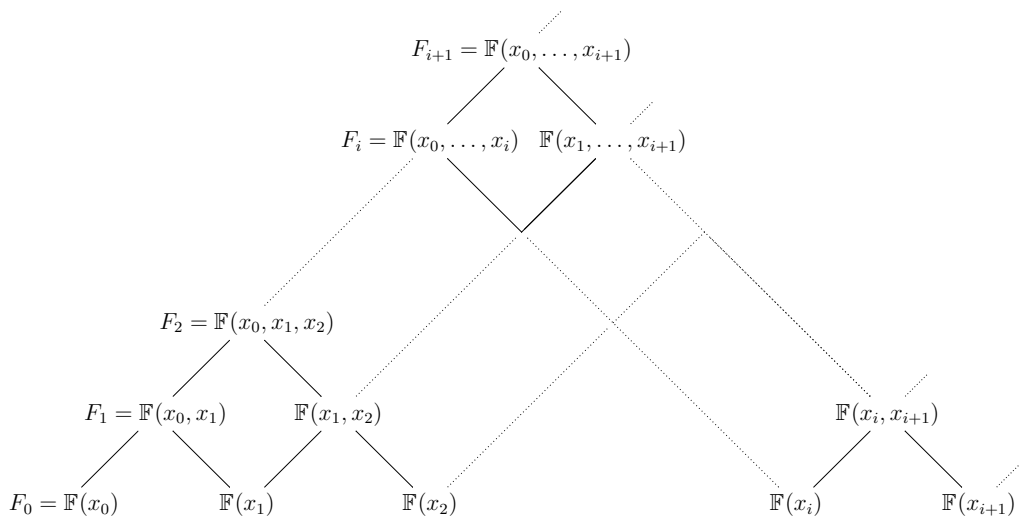
\begin{figure}[h!]
\begin{center}
\resizebox{\textwidth}{!}{
\makebox[\width][c]{
\xymatrix@!=2.5pc@dr{
F_{i+1}=\F{}(x_0,\ldots, x_{i+1}) \ar@{-}[r] \ar@{.}[];[]+/ur 1cm/ & \F{}(x_1,\ldots, x_{i+1})  \ar@{.}[rrr] \ar@{-}[];[d]+0 \ar@{.}[];[]+/ur 1cm/ &  \ar@{.}[]+0;[d]+0 \ar@{.}[rr]&&\F{}(x_i,x_{i+1})\ar@{-}[r]\ar@{.}[];[]+/ur 1cm/& \F{}(x_{i+1}) \ar@{.}[];[]+/ur 1cm/ \\
F_{i}=\F{}(x_0,\ldots, x_{i}) \ar@{-}[];[r]+0 \ar@{-}[u]&\ar@{-}[u] \ar@{.}[rrr]& && \F{}(x_i) \ar@{-}[u]\\
&&&&\\
F_2=\F{}(x_0,x_1,x_2) \ar@{-}[r] \ar@{.}[uu]& \F{}(x_1,x_2) \ar@{.}[uu] \ar@{-}[r]&\F{}(x_2) \ar@{.}[uu]\\
F_1=\F{}(x_0,x_1) \ar@{-}[u] \ar@{-}[r] & \F{}(x_1)\ar@{-}[u]\\
F_0=\F{}(x_0) \ar@{-}[u]
}
}
}
\end{center}
\caption{The pyramid corresponding to a recursive tower over a field $\F{}$.\label{fig:pyramid}}
\end{figure}
A recursive tower of function fields over $\F{q}$ in general, is simply a tower $\mathcal F=(F_n)_{n \ge 0}$ for which there exists a bivariate polynomial $f(X,Y) \in \F{q}[X,Y]$ such that $F_n=\F{q}(x_0,\dots,x_n)$ and $f(x_{i-1},x_i)$ for $i=1,\dots,n$. Key properties of the tower $\mathcal F$ can then be obtained by studying the behaviour of the function field extensions $F_1/F_0$ and $F_1/\F{q}(x_1)$. Once these two extensions are understood, the first layer of the pyramid in Figure \ref{fig:pyramid} is understood. Then the question is to which extent this information can be used to understand consecutive layers of the pyramid, eventually obtaining information about the tower $\mathcal F$ itself. The existence of a nonempty splitting locus in particular can be studied in this way quite nicely, compare for example with \cite[Prop. 7.2.20]{Stich09}. It was for example shown in this way that for a large class of towers, the splitting locus cannot increase if the constant field is increased \cite{B04}. The techniques from \cite{B04} were further developed in \cite{HP14}. It is suspected, but not known, that a good recursively defined tower must have a rational splitting place. As indicated previously, this splitting place does not necessarily have to be a place of $F_0$, but could be a rational place of $F_i$ for some $i \ge 0$ splitting totally in the extension $F_n/F_i$ for each $n > i$. Using techniques from \cite{B04} it was shown in \cite[Thm. 4.10]{BGS05ii} that under some technical conditions involving the number of places above the ramification locus, a good recursively defined tower will have a rational splitting place in $F_0$, but the general case is still open.

Also the finiteness of the ramification locus can be reformulated in terms of properties of the ``first layer'' of the pyramid, compare for example with \cite[Prop. 7.2.23]{Stich09}. It is more tricky though to investigate if the ramification remains $b$-bounded when passing from the first layer to higher layers of the pyramid. Of course if all ramification in the extensions $F_1/F_0$ and $F_1/\F{q}(x_1)$ is tame, applying Abhyankar's lemma iteratively implies that also all ramification in $F_n/F_0$ is tame for any $n \ge 1$. In this case all ramification is $1$-bounded. If wild ramification occurs, the situation is more complicated. Indeed, tower $\mathcal D$ introduced before has a finite ramification locus, but ramification is not $b$-bounded in the tower for any $b$. It is suspected, but not known, that a good recursively defined tower must have a finite ramification locus and that all ramification needs to be $b$-bounded for some $b$.

A second recursively defined tower was found in 1996 in \cite{GS96}. Also this tower is optimal, and defined recursively as follows: $F_n=\F{q^2}(x_0,\dots,x_n)$, where
\begin{equation}\label{eq:JNTtower}
x_{i}^q+x_{i}=\frac{x_{i-1}^q}{x_{i-1}^{q-1}+1} \quad \text{for $i=1,\dots,n$}.
\end{equation}
Like the previous tower, it has a nonempty splitting locus, $\F{q^2} \setminus \{\alpha \mid \alpha^q+\alpha=0\}$ and finite ramification locus $\{\alpha \mid \alpha^q+\alpha=0\} \cup \{\infty\}$. Like in \cite{GS95}, the main difficulty in the proof of optimality given in \cite{GS96}, is the genus computation of the function fields $F_n$. In this case, in hindsight, Garcia and Stichtenoth showed that all ramification in the tower $\mathcal F$ is $2$-bounded. Later \cite{GS05}, the same authors found  a much simpler proof of optimality, based on the following result, which we prove for the convenience of the reader:
\begin{lemma}[\cite{GS05}]\label{lem:2bounded}
Suppose that $F$ is a function field with characteristic $p$ and that two linearly disjoint extensions $F_1/F$, $F_2/F$ are given, both Artin--Schreier extensions of prime degree $p$. Denote by $F'$ the composite of $F_1$ and $F_2$.
Let $P'$ be a place of $F'$ and denote by $P_1,P_2,P$ its restriction to $F_1,F_2,F$. Further assume $e(P_1|P)=e(P_2|P)=p$, $d(P_1|P)=d(P_2|P)=2(p-1).$ Then either $e(P'|P_1)=e(P'|P_2)=1$ or $e(P'|P_1)=e(P'|P_2)=p$ and $d(P'|P_1)=d(P'|P_2)=2(p-1).$
\end{lemma}
\begin{proof}
Since $F'/F$ is a Galois extension of degree $p^2$, it is clear that either $e(P'|P_1)=e(P'|P_2)=1$ or $e(P'|P_1)=e(P'|P_2)=p.$ Assume the latter and that $F_i=F(y_i)$, with $y_i^p-y_i=f_i$ for some $f_i \in F$, so that the situation is as in Figure \ref{fig:diamond}.
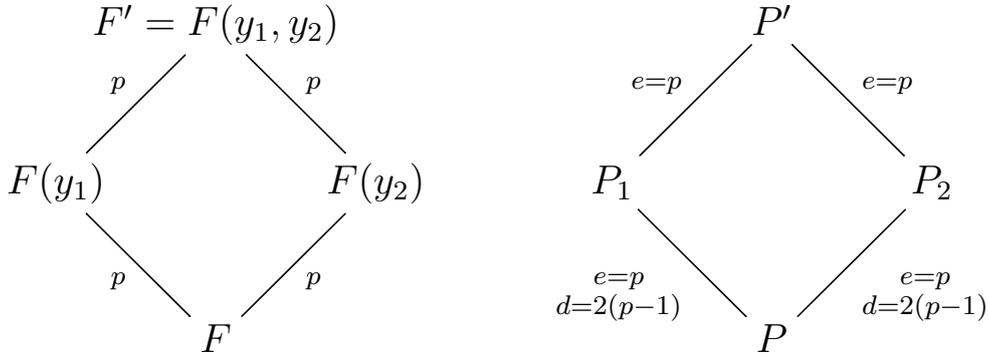
\begin{figure}[h!]
\begin{center}
\resizebox{\textwidth}{!}{
\makebox[\width][c]{
\xymatrix@!=2.5pc@dr{
F^\prime=F(y_1,y_2) \ar@{-}[d]_p \ar@{-}[r]^p & F(y_2)\\
F(y_1) \ar@{-}[r]_p& F \ar@{-}[u]_p\\
}

\hspace{1cm}

\xymatrix@!=2.5pc@dr{
P^\prime \ar@{-}[d]_{e=p} \ar@{-}[r]^{e=p} & P_2\\
P_1 \ar@{-}[r]_{\substack{e=p \\d=2(p-1)}}& P \ar@{-}[u]_{\substack{e=p \\ d=2(p-1)}}\\
}

}
}
\end{center}
\caption{overview of the situation\label{fig:diamond}}
\end{figure}
Since $G:=\mathrm{Gal}(F'/F)=\Z/p\Z \times \Z/p\Z,$ all $p+1$ intermediate fields $E$ of $F'/F$ with $[E:F]=p$ are given by the fields $F_1$ and the $p$ fields $F(y)$, with $y^p-y=\alpha f_1+f_2,$ with $\alpha \in \F{p}.$ Note that $d(P_1|P)=d(P_2|P)=2(p-1)$ implies that we may assume that $v_P(f_1)=v_P(f_2)=-1.$ Hence for all $\alpha \in \F{p}$, $v_P(\alpha f_1+f_2) \ge -1$, but since we assumed $e(P'|P_1)=e(P'|P_2)=p,$ strict inequality cannot occur. Hence if $E$ is any non-trivial intermediate field of $F'/F$ and $Q$ is the restriction of $P'$ to $E$, then $d(Q|P)=2(p-1)$.

Let us denote by $G_i$ the $i$-th ramification group of $P'|P$. Now define $n_1$ to be the maximum integer $i$ such that $|G_i|=p^2$ and $n_2$ the largest integer such that $|G_i| \ge p$:
$$\underbrace{G_0 = \cdots = G_{n_1}}_{\Z/p\Z \times \Z/p\Z} \supsetneq \underbrace{G_{n_1+1} = \cdots = G_{n_2}}_{\Z/p\Z} \supsetneq \{\mathrm{id}\}.$$
Note that if none of the ramification groups has cardinality $p$, then by definition $n_1=n_2$. In this case the part $G_{n_1+1} = \cdots = G_{n_2}$ in the previous equation can be left out. 
Either way, $d(P^\prime|P)=(n_1+1)(p^2-1)+(n_2-n_1)(p-1)$ by Hilbert's different theorem. Also note $n_1 \ge 1,$ since $|G_1|=e(P'|P)=p^2$.
If $E=\mathrm{Fix}(G_{n_2})$, then $d(P^\prime|Q)=(n_2+1)(p-1)$, whence $d(Q|P)=(d(P^\prime|P)-d(P^\prime|Q))/p=(n_1+1)(p-1)$ using transitivity of the different. Hence $n_1=1$ from the previous. Applying a similar computation for any other intermediate field now gives $n_2=n_1=1$. Hence $d(P^\prime|Q)=2(p-1)$ for any nontrivial intermediate field of $F'/F$.
\end{proof}
For a specific type of Artin--Schreier extensions, a similar observation as in this lemma was made in \cite{LSY05}. Returning to the tower $\mathcal F$ from \cite{GS96} in case $q=p$, this result immediately implies that if ramification in the extensions $F_1/F_0$ and $F_1/\F{q^2}(x_1)$ is $2$-bounded, then the same is true for ramification in $F_n/F_1$ for all $n \ge 1$. One simply uses the pyramid of function fields corresponding to $\mathcal F$ and uses Lemma \ref{lem:2bounded} iteratively. Showing $2$-boundedness of the ramification in $F_1/F_0$ and $F_1/\F{q^2}(x_1)$ is not hard and can be done by a direct computation. Also if $q$ is not a prime, a similar reasoning applies, since the extensions $F_1/F_0$ and $F_1/\F{q^2}(x_1)$ can be divided into degree $p$ Artin--Schreier extensions, using elementary Galois theory. All in all, this shows that:
$$\lim_{n \to \infty} \frac{N_1(F_n)}{g(F_n)} \ge \frac{S(\mathcal F)}{(-2+2\sum_{P \in V(\mathcal F)} \deg(P))/2}=q-1.$$
Hence like tower $\mathcal E$, also tower $\mathcal F$ is optimal over $\F{q^2}.$

\begin{remark}
Now that the optimality of tower $\mathcal F$ has been established using Lemma \ref{lem:2bounded}, it turns out to be relatively simple to prove optimality for the tower $\mathcal E$ as well. What was missing in the previous was a proof of the fact that all ramification in $E_n/E_0$ is $(q+2)$-bounded. Considering equation \eqref{eq:INV95} and writing $z_i:=x_{i-1}x_i$, we obtain that for $i \ge 1$:
$$z_{i+1}^q+z_{i+1}=x_{i}^{q+1}=\frac{z_{i}^{q+1}}{x_{i-1}^{q+1}}=\frac{z_{i}^{q+1}}{z_i^q+z_i}=\frac{z_{i}^{q}}{z_i^{q-1}+1}.$$
This was observed in \cite[Remark 3.11]{GS96}. Hence identifying $F_{n-1}$ with the subfield $\F{q^2}(z_1,\dots,z_n)$ of $E_n$, we see that $E_n=F_{n-1}(x_0).$ On the other hand $x_0^{q+1}=z_1^q+z_1$, so the field extensions $E_n/F_{n-1}$ are Kummer extensions for all $n \ge 1$. Using that all ramification in the tower $\mathcal F$ is $2$-bounded, it is then not hard to show that all ramification in the tower $\mathcal E$ is $(q+2)$-bounded.
\end{remark}

\subsection{Some further good towers}

In the previous subsection, we have reviewed the first two Garcia--Stichtenoth towers, both optimal over $\F{q^2}$ for any prime power $q$. In the years after the discovery of these two towers, several other good towers were found. First of all, in 1997 Garcia, Stichtenoth, and Thomas found a number of good towers in \cite{GST97} of which we would like to single out the following: $\mathcal G=(G_n)_{n \ge 0}$, where $G_n=\F{p^e}(x_0,\dots,x_n)$ and
\begin{equation}\label{eq:GST}
x_{i}^{\frac{p^e-1}{p-1}}=-(x_{i-1}+1)^{\frac{p^e-1}{p-1}}+1 \quad \text{for $i=1,\dots,n$}.
\end{equation}
The ramification locus of $\mathcal G$ can be seen to be contained in $\F{p^e}$, while the splitting locus equals $\{\infty\}$. Since all ramification is tame, this is enough to conclude that $\lambda(\mathcal G) \ge 2/(p^e-2)$. In other words: it was shown in \cite{GST97} that for $e>1$,
\begin{equation}\label{eq:GSTbound}
A(p^e) \ge 2/(p^e-2).
\end{equation} 
Though in general this bound is weaker than the logarithmic bound by Serre, for small $p$ and $e$ it can be better. In case $p=e=2$, the tower $\mathcal G$ is even optimal. Note though that for $e=3$, the bound $A(p^3) \ge 2/(p^3-2)$ is weaker than Zink's bound.

In 2002, G.~van der Geer and M.~van der Vlugt \cite{GV02} found a recursive tower over $\F{8}$ whose limit is $3/2$, exactly the value of Zink's bound for $q=8$. The recursive equation they found was $x_i^2+x_i=x_{i-1}+1+1/x_{i-1}.$ Similarly as in \cite{GS95,GS96}, it is not hard to see that the splitting locus is nonempty, namely $\F{8}\setminus\{0,1\}$ and that the ramification locus is finite, namely $\F{4} \cup \{\infty\}$. The hard part is again the genus computation. However, using Lemma \ref{lem:2bounded}, the limit of the tower can be obtained in an easier way as observed in \cite{GS05}. Indeed, it turns out that all ramification in the function field extensions $\F{8}(x,y)/\F{8}(x)$ and $\F{8}(x,y)/\F{8}(y)$, with $y^2+y=x+1+1/x$, is $2$-bounded. Applying Lemma \ref{lem:2bounded} implies that all ramification in the vdGeer--vdVlugt tower is $2$-bounded, and we indeed find that its limit is at least $3/2$.

Motivated by the existence of a recursive tower attaining Zink's bound for $p=2$, a search for similar towers attaining Zink's bound for other values of $p$ started. This eventually resulted in the tower found by Bezerra, Garcia, and Stichtenoth in \cite{BezGS05}. It is recursively defined by the equation $$\frac{1-x_{i}}{x_i^q}=\frac{x_{i-1}^q+x_{i-1}-1}{x_{i-1}},$$  and it is shown in \cite{BezGS05} that the resulting tower has limit at least equal to $2(q^2-1)/(q+2)$, showing that for any prime power $q$ it holds that $A(q^3) \ge 2(q^2-1)/(q+2)$.
Making a slight change of variables, $y_i=1/x_i$, one obtains $$y_i^q-y_i^{q-1}=-y_{i-1}+1+1/y_{i-1}^{q-1},$$ which shows how the vdGeer--vdVlugt tower arises as a special case for $q=2$. Unless $q=2$, the extensions $\F{q^3}(x_0,x_1)/\F{q^3}(x_0)$ and $\F{q^3}(x_0,x_1)/\F{q^3}(x_1)$ are not Artin--Schreier and in fact not even Galois. This made it necessary to do extensive genus computations in order to find the limit of the tower. However, a few years later in \cite{BaGS08}, a tower with Galois steps and limit $2(q^2-1)/(q+2)$ was found where Lemma \ref{lem:2bounded} can be applied to avoid lengthy and technical genus computations. Details of this tower are also given in the second part of section 7.4 in \cite{Stich09}.
More good recursively defined towers were found and apart from the already cited articles, another source of examples is \cite{GSR03}. The towers in \cite{GSR03} have the feature that all ramification is tame. Additional references to articles containing examples of good towers will be given in the next sections.

\section{Recursively defined towers and modular curves}

So far, we have discussed various methods to obtain families of curves or function fields over a finite field with a view to finding lower bounds for Ihara's constant. Two methods, using modular curves and using recursively defined towers of function fields, can be used to prove Ihara's result that $A(q^2) \ge q-1$. In 1998 Elkies showed in \cite{E98} that certain families of classical modular curves can be described recursively. This article provides various explicit examples of recursive, optimal towers as well. In this way, Elkies gave a modular interpretation to several explicit recursively defined optimal towers as well. Later, he also showed that the first two Garcia--Stichtenoth towers have a modular interpretation in \cite{E01}, this time using Drinfeld modular curves. In a further work \cite{LMSE02}, several optimal recursive towers were found and in the same article a modular interpretation was given. Like Elkies did in \cite{E98}, one may wonder if ``perhaps every asymptotically optimal tower of this recursive form must be modular?''. In this section we give an overview of this interesting connection.

\subsection{A recursive description of towers of modular curves}

Let $N>1$ be an integer. As indicated previously, the function field of the classical modular curve $X_0(N)$ can be described as $\Q(x_0,x_1)$, where $\Phi_N(x_0,x_1)=0$. Moreover, a pair of $N$-isogenous elliptic curves $(E_0,E_1)$ defined over $\C$, or rather an isomorphism class of such pairs, gives rise to a point on $X_0(N)$. Let us denote by $\phi:E_0 \to E_1$, the corresponding cyclic isogeny of degree $N$. If $\phi_1: E_0 \to E_1$ and $\phi_2: E_1 \to E_2$ are two cyclic isogenies of degree $N$, then $\phi_2 \circ \phi_1: E_0 \to E_2$ is an isogeny of degree $N^2$. Note that $\phi_2 \circ \phi_1$ needs not be a cyclic isogeny. On the other hand, any cyclic isogeny $\psi: E_0 \to E_2$ of degree $N^2$ can be ``factored'' as the composition $\phi_2 \circ \phi_1$ of two cyclic isogenies $\phi_1: E_0 \to E_1$ and $\phi_2:E_1 \to E_2$ both of degree $N$ and a suitably chosen elliptic curve $E_1$. The maps $\phi_1,\phi_2$ and elliptic curve $E_1$ are unique up to isomorphism. Indeed, the kernel of $\phi_1$ is the unique cyclic subgroup $C_N$ of $\ker(\psi)$ of order $N$, while $E_1 \cong E_0/C_N$. In function field terms: the function field of $X_0(N^2)$ can be described as a suitable compositum of two copies of the function field of $X_0(N)$. More generally, if $\phi: E_0 \to E_{n}$ is a cyclic isogeny of degree $N^n>1$, there exist cyclic isogenies $\phi_1,\dots,\phi_n$ of degree $N$ and elliptic curves $E_1,\dots,E_{n-1}$ such that $\phi=\phi_n \circ \cdots \circ \phi_1$.
\begin{equation}\label{eq:modulareqs}
\xymatrix{
  E_0 \ar[r]^{\phi} & E_n & \text{gives rise to} & E_0 \ar[r]^{\phi_1} & E_1 \ar[r]^{\phi_2} & \cdots \ar[r]^{\phi_n} & E_n
}
\end{equation}
Hence, the corresponding point on $X_0(N^n)$ gives rise to an $n$-tuple of points on $X_0(N)$. In function field terms: the function field of $X_0(N^n)$ can be obtained as a suitable composite of $n$ copies of the function field of $X_0(N)$. This is the essence of one of the observations Elkies made in \cite{E98}. We will describe how to obtain explicit equations describing this composite in a slightly different way as in \cite{E98}, but the essential ingredients are the same. Exactly the same approach works for Drinfeld modular curves, which was the point of view taken in \cite{BBN14}. Indeed, the exposition given below follows the one in \cite{BBN14} closely.

Using $j$-invariants to describe points on $X_0(N^n)$, let us denote the $j$-invariants of the elliptic curves $E_i$ in equation \eqref{eq:modulareqs} by $j_i$. Then we know that $\Phi_{N^n}(j_0,j_n)=0$ and indeed $\Q(x_0,x_n)$, with $\Phi_{N^n}(x_0,x_n)=0$, is the function field of $X_0(N^n).$ However, we also know that $\Phi_N(j_{i-1},j_i)=0$ for $i=1,\dots,n$, so a first guess may be to recursively describe $\Q(x_0,x_n)$ as $\Q(x_0,x_1,\dots,x_n)$, where $\Phi_N(x_{i-1},x_i)=0$ for $i=1,\dots,n$. There is a problem though, which already occurs for $n=2$: the polynomial $\Phi_N(x_1,T) \in \Q(x_0,x_1)[T]$ is not irreducible. Indeed, since the modular polynomials $\Phi_N(X,Y)$ are symmetric, $\Phi(x_1,T)$ has the root $T=x_0$ in $\Q(x_0,x_1)$. The reason for this goes back to the modular properties of $X_0(N)$. Given a pair $(E_0,E_1)$ of $N$-isogenous elliptic curves with isogeny $\phi_1: E_0 \to E_1$, one may form the dual isogeny $\phi_2:=\widehat{\phi}: E_1 \to E_0$. Then the composite $\phi_2 \circ \phi_1$ equals the multiplication by $N$ map, whose kernel is not cyclic, but isomorphic to $\Z/N\Z \times \Z/N\Z$. More generally, factors of $\Phi_N(x_1,T) \in \Q(x_0,x_1)[T]$ correspond to the situation where the kernel of $\phi_2 \circ \phi_1$ is not cyclic, but is isomorphic to some other abelian subgroup of $\Z/N^2\Z \times \Z/N^2\Z$ of order $N^2$ and exponent strictly less than $N^2$. Fortunately, there is an easy way out of this: the function field extension $\Q(x_0,x_1,x_2)/\Q(x_0,x_1)$, should geometrically correspond to the natural projection $\pi_2: X_0(N^2) \to X_0(N)$, where a point coming from a pair of $N^2$-isogenous elliptic curves $(E_1,E_2)$ with $N^2$-cyclic isogeny $\psi$, is mapped to the point of $X_0(N)$ coming from the pair $(E_0,E_0/C_N)$, $C_N$ being the unique cyclic subgroup of $\ker(\psi)$ of order $N$. However, the map $\pi_2$ has degree $N$. This is essentially because there are $N$ ways to extend a cyclic subgroup $C_N$ of $E_0[N^2] \cong \Z/N^2\Z \times \Z/N^2\Z$ of order $N$, to a cyclic subgroup of order $N^2$. Each such extended subgroup can then act as the kernel of an $N^2$-cyclic isogeny from $E_0$ to another elliptic curve. Similarly, for any $n \ge 2$, the degree of the natural projection map $\pi_n: X_0(N^n) \to X_0(N^{n-1})$ is $N$, for details see \cite{E98}.

The above discussion shows that the field extension $\Q(x_0,x_1,x_2)/\Q(x_0,x_1)$ should have degree $N$, so to find equations for it, all we need to do is to choose the irreducible factor of $\Phi_N(x_1,T) \in \Q(x_0,x_1)[T]$ of degree $N$, say $\Psi_N(x_0,x_1,T)$. Then the function field of $X_0(N^2)$ can be described as $\Q(x_0,x_1,x_2)$, where $\Phi_N(x_0,x_1)=0$ and $\Psi_N(x_0,x_1,x_2)=0$. More generally, for $n \ge 2$ the function field of $X_0(N^n)$ is then given as $\Q(x_0,x_1,\dots,x_n)$, where $\Phi_N(x_0,x_1)=0$ and $\Psi_N(x_{i-2},x_{i-1},x_i)=0$ for $i=2,\dots,n$. All in all, we have shown that the tower $\mathcal F^{(N)}$ of function fields corresponding to the tower of modular curves
\begin{equation*}
\xymatrix{
\cdots \to^{\pi_{n+1}} X_0(N^n) \to^{\pi_{n}} X_0(N^{n-1}) \to^{\pi_{n-1}} \cdots \to^{\pi_2} X_0(N) \to^{\pi_1} X(1)
}
\end{equation*}
can be recursively defined as:

\begin{enumerate}
\item[] $F_0^{(N)}=\Q(x_0)$,
\item[] $F_1^{(N)}=\Q(x_0,x_1)$ with $\Phi_N(x_0,x_1)=0$, and
\item[] $F_n^{(N)}=F_{n-1}^{(N)}(x_n)$ with $\Psi_N(x_{n-2},x_{n-1},x_n)=0$, for $n \ge 2$.
\end{enumerate}

Note that the depth of the recursion is actually two, not one. To obtain an optimal, recursively defined tower of function fields over $\F{p^2}$, with $p$ not dividing $N$, one should simply reduce all equations modulo $p$, then extend the constant field from $\F{p}$ to $\F{p^2}$.

\subsection{Modular interpretation of some optimal recursive towers}

As an example, let us consider the optimal tower $\mathcal G$ over $\F{4}$ recursively defined using the equation $x_i^3=-(x_{i-1}+1)^3+1.$ Since we work in characteristic two, we can rewrite this as $x_i^3=x_{i-1}^3+x_{i-1}^2+x_{i-1}$. This is a special case of a recursive tower from \cite{GST97} given in equation \eqref{eq:GST}. We will see that this tower can be identified with the function fields of the tower of modular curves when reduced modulo $2$:
\begin{equation*}
\xymatrix{
\cdots \to^{\pi_{n+1}} X_0(3^n) \to^{\pi_{n}} X_0(3^{n-1}) \to^{\pi_{n-1}} \cdots \to^{\pi_3} X_0(3^2)
}
\end{equation*}
We start by following the recipe explained in the previous subsection. First of all the modular polynomial $\Phi_3(X,Y)$ when reduced modulo $2$ is equal to $Y^4+X^3Y^3+X^4$. This means that $\F{4}(X_0(3))$, the function field of $X_0(3)$ reduced modulo $2$, is $\F{4}(X_0(3))=\F{4}(j_0,j_1)$, where $j_1^4+j_1^3j_0^3+j_0^4=0$. Note $\F{4}(j_0,j_1)=\F{4}(s_0)$, with $s_0=j_1/(j_0+j_1)+1/j_0+1/j_1$, since then 
$$1/j_0=s_0^2(s_0^2+s_0) \text{ and } 1/j_1=(s_0^2+1)(s_0^2+s_0).$$
The reason such $s_0$ exists is that the genus of $X_0(3)$ is zero.

Now for $i \ge 0$ define $s_i:=j_{i+1}/(j_i+j_{i+1})+1/j_i+1/j_{i+1}.$  Then $\F{4}(X_0(3^2))=\F{4}(j_0,j_1,j_2)=\F{4}(s_0,s_{1})$. We have $s_1^2(s_1^2+s_1)=1/j_1=(s_0^2+1)(s_0^2+s_0)$ and
\begin{multline*}
s_1^2(s_1^2+s_1)-(s_0^2+1)(s_0^2+s_0)= \\ (s_1+s_0+1)(s_0^3 + s_0^2s_1 + s_0s_1^2 + s_1^3+s_0s_1 + s_0).
\end{multline*}
Since $[\F{4}(X_0(3^2)):\F{4}(X_0(3))]=3$, can conclude that
$$s_0^3 + s_0^2s_1 + s_0s_1^2 + s_1^3+s_0s_1 + s_0=0.$$
More generally, $\F{4}(X_0(3^n))=\F{4}(j_0,\dots,j_{n})=\F{4}(s_0,\dots,s_{n-1})$ for $n \ge 2$, and a full set of defining equations is given by
$$s_{i-1}^3 + s_{i-1}^2s_i + s_{i-1}s_i^2 + s_i^3+s_{i-1}s_i + s_{i-1}=0 \quad \text{for $i=1,\dots,n-1$}.$$
This time, we did not need a depth two recursion, but a depth one recursion. The reason for this was that because $X_0(3)$ has genus zero, this allows one to recast the defining equations in terms of a uniformizer of $\F{4}(X_0(3))$.

In fact, also the curve $X_0(3^2)$ has genus zero, which allows for a further simplification of the defining equations. Indeed, $\F{4}(X_0(9))=\F{4}(s_0,s_1)=\F{4}(x_0)$ for $x_0=s_1/(s_1+s_0+1)$ in which case $x_0^2+x_0=s_1+s_0$, implying $s_1=x_0^3+x_0^2+x_0$ and $s_0=x_0^3$. Defining $x_i:=s_{i+1}/(s_{i+1}+s_i+1)$, we see that $\F{4}(X_0(3^n))=\F{4}(x_0,\dots,x_{n-2})$ for $n \ge 2$, where $x_i^3=s_i=x_{i-1}^3+x_{i-1}^2+x_{i-1}$ for $i=0,\dots,n-2$. This was exactly the equation we started out with identifying the function field of the tower $\mathcal G$ with $\F{4}(X_0(3^n))$ with $n \ge 2$. Note that the fact that $X_0(3^2)$ has genus zero, allows one to recast the defining equations in a uniformizer of $\F{4}(X_0(3^2))$, resulting in an equation where the variables can be separated.

The cases where Elkies worked out explicit defining equations for towers of classical modular curves are mainly those of the form $(X_0(m \ell^n))_{n \ge 2},$ where $m \in \Z_{\ge 1}$ and $\ell \in \Z_ {\ge 2}$ are integers such that both $X_0(m\ell)$ and $X_0(m \ell^2)$ have genus zero. As above, this allows one to find recursive equations for the tower of function fields $\Q(X_0(m \ell^n))_{n \ge 2}$ where the variables can be separated. The case $m=2$, $\ell=3$ was not included in \cite{E98}, but details for that case have been worked out in \cite[Example 5.6]{BBo05}. The modular interpretation for the first two Garcia--Stichtenoth is based on a recursive description of the Drinfeld modular curves $X_0(T^n)$. Since both curves $X_0(T)$ and $X_0(T^2)$ have genus zero, which follows for example from the genus formulas in \cite{Gekeler86}, one can find a depth one recursive description where the variables can be separated, rather than the depth two recursive description one would expect in general.

\subsection{Two further examples of optimal recursive towers}

In \cite{E98,E01,LMSE02} a modular interpretation was given to all optimal, recursive towers that were known at that time. Since then, more optimal towers were found. See for example \cite{MW05}, though it was indicated in this article that all found optimal examples dominate a modular tower, suggesting modularity for all these examples. In this subsection, we would like to give two further examples of optimal, recursive towers. The first one is taken from \cite{loetter}, the second one from \cite{CNT18}.

E.C.~L\"otter used a computer search to find several examples of good recursive towers. One of these is the tower over $\F{7^4}$ recursively defined by the equation
$$x_i^5=\frac{x_{i-1}^5+5x_{i-1}^4+x_{i-1}^3+2x_{i-1}^2+4x_{i-1}}{2x_{i-1}^4+5x_{i-1}^3+2x_{i-1}^2+x_{i-1}+1}.$$
It will be convenient to perform a slight change of variables by replacing $x_{i-1}$ by $3x_{i-1}$ and $x_i$ by $3x_i$, resulting in the equation:
\begin{equation}\label{eq:loetter}
x_i^5=x_{i-1} \frac{1-2x_{i-1}+4x_{i-1}^2-3x_{i-1}^3+x_{i-1}^4}{1+3x_{i-1}+4x_{i-1}^2+2x_{i-1}^3+x_{i-1}^4}.
\end{equation}
It was shown in \cite{loetter} that this tower over $\F{7^4}$ has limit at least $6$. A modular interpretation of this tower was given in \cite{BBN14}, where it was showed that if equation \eqref{eq:loetter} is used to define a tower over $\Q$, then this tower can be obtained by a degree two extension of the tower of function fields of the modular curves $(X_0(5^n))_{n \ge 2}$. In essence, the ``base'' of the second tower is shifted from the function field of $X_0(5)$ to that of the modular curve $X(5)$. As observed in \cite{BBN14}, equation \eqref{eq:loetter} occurs in the famous first letter that S.~Ramanujan wrote more than a 100 years ago to G.H.~Hardy. In it, Ramanujan defined a continued fraction, now known as the Rogers--Ramanujan continued fraction, and related two of its values exactly by equation \eqref{eq:loetter}, see \cite[Theorem 5.5]{BCHKSS99} for details. This continued fraction can be seen as a generator of the function field of the modular curve $X(5)$, which has genus zero. This explains why $X(5)$ occurs in the modular interpretation. To obtain a good tower over a finite field, one needs to reduce equation \eqref{eq:loetter} modulo a prime $p$ distinct from $5$. However, where for the modular curves $X_0(5^n)$ one then needs to extend the constant field to $\F{p^2}$ in order to obtain a curve with many rational points, for $X(5)$ one needs to consider a quadratic extension of the finite field $\F{p}(\zeta_5)$, where $\zeta_5$ is a fifth root of unity. This explains why for $p=7$, one needs to consider $\F{7^4}$ as constant field to get a good tower. It turns out that for general $p$ distinct from $5$, the limit of the tower is $p-1$ over a quadratic extension of $\F{p}(\zeta_5)$. This means that the tower is optimal if $p \equiv \pm 1 \pmod{5}$ and good if $p \equiv \pm 2 \pmod{5}$, see \cite{BBN14} for further details.
Interestingly enough, the right-hand side of equation \ref{eq:loetter} occurs exactly in that form in \cite{CDV20} as well, this time as a part of an explicit formula for a cyclic isogeny of degree $5$ between two elliptic curves in Tate normal form. In the same article, cyclic isogenies of other degrees are considered as well, in particular of degree $4$. One could similarly use the obtained degree $4$ formulas from \cite{CDV20} to define a recursive tower. In this case it turns out that one obtains the modular tower $X_0(4^n)_{n \ge 2}$, which quickly can be verified from the explicit equation given for this tower in \cite{E98}.

A second, more recently found optimal tower over $\F{4}$ given in \cite{CNT18}, is defined recursively by the equation
\begin{equation}\label{eq:CNTtower}
x_i^2+x_i=\frac{x_{i-1}}{x_{i-1}^2+x_{i-1}+1}.
\end{equation}
That this particular equation might be interesting was the outcome of a series of papers giving necessary condition for a recursive tower to be good \cite{GS00,BGS04,BGS05,BGS06}. More precisely, in \cite{BGS06} all possibly good recursive Artin--Schreier towers, towers where each step in the tower is an Artin--Schreier extension, were classified. In case the defining equation is assumed to have coefficients from the field $\F{2}$, only four possible equations remained. One of these gives rise to the second Garcia--Stichtenoth tower for $q=2$, that is equation \eqref{eq:JNTtower}, two of these give rise to the vdGeer--vdVlugt tower, one time with the roles of $x_i$ and $x_{i-1}$ reversed, and the final fourth one is equation \eqref{eq:CNTtower}. The contribution of \cite{CNT18} is to show that also this equation recursively defines an optimal tower over $\F{4}$. This times splitting locus and ramification locus are not distinct, but interact with each other. Currently no modular interpretation is known for this tower and no relation with other optimal towers is known. It would be interesting to see if also this optimal recursive tower can be explained using modular theory.

\section{Recursive towers of function fields: non-square finite fields}

If $q=p^{m}$ is a square, we have seen that $A(q)=\sqrt{q}-1=p^{m/2}-1$ and that families of function fields over $\F{q}$ attaining this, can be constructed using modular curves, be it either classical modular curves, Shimura curves, or Drinfeld modular curves. The first two Garcia--Stichtenoth towers, though given by a completely independent construction, were shown by Elkies to be Drinfeld modular. All in all, it is clear that using modular curves in various forms, the case that $q$ is a square can be settled completely. The case of non-square finite fields is currently still open. In this section, we discuss results from \cite{BBGS15}, where a particular recursive tower of function fields over $\F{p^m}$, $m>1$ any integer, was introduced to show that:
\begin{equation}\label{eq:non-square bound}
A(p^m)\ge 2\left(\frac{1}{p^{\lfloor m/2 \rfloor}-1}+\frac{1}{p^{\lceil m/2 \rceil}-1}\right)^{-1}.
\end{equation}
Here $\lfloor m/2 \rfloor$, resp. $\lceil m/2 \rceil$, denotes $m/2$ rounded down, resp. rounded up.

\subsection{A tower over a non-prime finite field}

Before discussing how to obtain equation \eqref{eq:non-square bound}, let us consider some special cases. First of all, if $m$ is even, equation \eqref{eq:non-square bound} simplifies significantly, since in that case $\lceil m/2 \rceil=\lfloor m/2 \rfloor=m/2$. Hence if $m$ is even, equation \eqref{eq:non-square bound} simplifies to $A(p^m) \ge p^{m/2}-1$, which is Ihara's result \cite{Ihara}. For $m=3$, equation \eqref{eq:non-square bound} can be rewritten as $A(p^3) \ge 2(p^2-1)/(p+2),$ which is exactly Zink's bound. Hence equation \eqref{eq:non-square bound} can be viewed as a common generalization of Ihara's results and the Zink bound. If one wants, one can rewrite equation \eqref{eq:non-square bound} as
$$A(p^m) \ge 2(p^{\lceil m/2 \rceil}-1)(p^{\lfloor m/2 \rfloor}-1)/(p^{\lceil m/2 \rceil}+p^{\lfloor m/2 \rfloor}-2).$$
The advantage of this is that, the right-hand side now is defined for $m=1$. On the other hand, in case $m=1$ is evaluates to zero. Of course the statement $A(p) \ge 0$ is trivially true, but uninteresting and Serre's logarithmic lower bound is better. One could interpret this as a sign that equality may not hold either for odd $m>1$ in equation \eqref{eq:non-square bound}, but opinions are divided on this. A conjecture that has been voiced more often is that perhaps for $m>1$ equation \eqref{eq:non-square bound} describes the best possible lower bound on $A(p^m)$ that can be obtained with recursively defined towers of function fields over $\F{p^m}.$ Be that as it may, equation \eqref{eq:non-square bound} is strong enough to disprove for all odd $m \ge 3$ the conjecture that $A(p^m)=p^{(m-1)/2}-1$ made in \cite{Manin}. Serre's logarithmic lower bound disproved the conjecture for $m=1$ and of course by Zink's bound it was already known that the conjecture could not hold for $m=3$ either. It is worth noting that for fixed $p$ and $m$ odd and tending to infinity, the lower bound for $A(p^m)$ from equation \eqref{eq:non-square bound} and the Drinfeld--Vladut upper bound are proportional to each other. Indeed, their ratio tends to $\frac{2 \sqrt{p}}{p+1}>0$ as $m$ odd tends to infinity.

Now, let us turn out attention to the results in \cite{BBGS15}. What was actually shown there is that for a given prime power $q$ and relatively prime, positive integers $j,k$ and $m:=j+k$, that any tower $\mathcal H=(H_n)_{n \ge 1}$ over $\F{q^m}$ satisfying the recursion
\begin{equation}\label{eq:BBGStower}
\frac{x_i^{q^m}-x_i}{x_i^{q^j}} = \frac{x_{i-1}^{q^m}-x_{i-1}}{x_{i-1}^{q^m-q^k+1}}
\end{equation}
has limit at least $2\left(\frac{1}{q^j-1}+\frac{1}{q^k-1}\right)^{-1}.$ There is a reason for using the terminology `satisfying the recursion' rather than `recursively defined by'. As observed in \cite{BBGS15}, the polynomial $(Y^{q^m}-Y)X^{q^m-q^k+1}  - Y^{q^j}(X^{q^m}-X)$ is reducible, but any factor of degree strictly larger than one, can be used to define a tower. Moreover, it does not matter which factor is chosen as far as the estimate for the limit is concerned. If $j=k=1$, whence $m=2$, one of these factors is in fact the polynomial $X^{q-1}Y^q+Y-X^q.$ Using this factor, one exactly recovers the first Garcia--Stichtenoth tower. Optimizing the bound from equation \eqref{eq:BBGStower}, one obtains equation \eqref{eq:non-square bound}.

Proving that tower $\mathcal H$ is good, was done in \cite{BBGS15} by observing that its splitting locus is $\F{q^m} \backslash \{0\}$, its ramification locus is $\{0,\infty\}$, while performing careful computations they showed that all ramification is bounded. More precisely, they show that for any place $P$ of $H_n$ lying above the zero $P_0$, resp. pole $P_\infty$, of $x_0$, it holds that
\begin{equation}\label{eq:boundednessb0}
d(P|P_0) \le \left(\frac{q^m-1}{q^k-1}+1 \right) (e(P|P_0)-1),
\end{equation} resp.
\begin{equation}\label{eq:boundednessbinf}
d(P| P_\infty) \le \left(\frac{q^m-1}{q^j-1}+1 \right) (e(P|P_\infty)-1).
\end{equation}
The lower bound on the limit then follows:
$$\lambda(\mathcal H) \ge \dfrac{q^m-1}{-1+\left(\frac{q^m-1}{q^j-1}+1 \right)/2+\left(\frac{q^m-1}{q^k-1}+1 \right)/2}=\dfrac{2}{\frac{1}{q^j-1}+\frac{1}{q^k-1}}.$$

\subsection{Modular explanation of the equations}

While with the first two Garcia--Stichtenoth towers, the equations were not found using modular theory and only later a modular explanation was found, in \cite{BBGS15}, simultaneously, the equations were introduced as well as motivated from a modular point of view. In this subsection, we would like to briefly sketch this modular interpretation of equation \eqref{eq:BBGStower}. In order to do this, we need to fix a few notations:
\begin{itemize}
\item $L$, a field containing $\F{q}$,
\item $\iota: \F{q}[T] \to L$, a fixed $\F{q}$-algebra homomorphism,
\item $\tau: L \to L$, the $q$-th power map
\item $L\{\tau\}$, the ring of additive polynomials over $L$ under operations of addition and composition 
\item $D: L\{\tau\} \to L$, defined by $a_0+a_1\tau+\cdots+a_n \tau^n \mapsto a_0$
\end{itemize}
With this in place, we can define what a Drinfeld module is.
\begin{definition}
A homomorphism of $\F{q}$-algebras $\phi:\F{q}[T] \to L\{\tau\}$ (where one usually writes $\phi_a$ for the image of $a$ under $\phi$) is called a \emph{Drinfeld module} $\phi$ over $L$, if $D\circ \phi=\iota$ and if there exists $a \in \F{q}[T]$ such that $\phi_a \neq \iota(a)$.
\end{definition}

The kernel of $\iota$ is called the \emph{characteristic} of the Drinfeld module. Since $\F{q}[T]$ is a PID and $\F{q}[T]/\ker(\iota)$ is isomorphic to a subring of $L$, one can in this setting also describe the characteristic of $L$ by an irreducible, monic polynomial $P(T) \in F{q}[T].$ Note that $\phi$ is determined by the additive polynomial $\phi_T$. If $\phi_T=g_m\tau^m+g_{m-1}\tau^{m-1}+\cdots+g_{1}\tau+g_0 \in L\{\tau\}$, with $g_i\in L$ and $g_m\neq 0$, the Drinfeld module is said to have rank $m$. In characteristic $P(T)$, the Drinfeld module $\phi$ is called supersingular if $\phi_{P(T)}$ is a purely inseparable additive polynomial. In the special case $P(T)=T-1$, this means $g_1=\cdots=g_{m-1}=0$. Finally, a $T$ torsion point is an element $x \in \overline{L},$ the algebraic closure of $L$, satisfying $g_mx^{q^m}+g_{m-1}x^{q^{m-1}}+\cdots+g_{1}x^q+g_0x=0.$

A last general, important notion is that of an isogeny between Drinfeld modules.
\begin{definition}
For Drinfeld modules $\phi$ and $\psi$ as above, an isogeny $\lambda:\phi\to \psi$ over $L$ is an element $\lambda\in L\{\tau\}$ satisfying
\begin{equation*}
\label{isogeny}
\lambda \cdot \phi_a=\psi_a \cdot \lambda \text{ for all } a\in \F{q}[T].
\end{equation*}
\end{definition}
All this is standard in the theory of Drinfeld modules, see \cite{Gekeler86} or \cite{Goss} for more details.

For $1 \le j \le m$, let $\mathfrak D_{m,j}$ be the one-parameter family of rank $m$ Drinfeld modules of characteristic $T-1$ of the form $\phi_T=-\tau^m+g\tau^{j} + 1$. We assume that $\gcd(m,j)=1$ and write $k=m-j$. Note that $\mathfrak D_{m,j}$ contains the supersingular Drinfeld module $\phi$ with $\phi_T=-\tau^m+1$. Let $x$ be a $T$-torsion point of the Drinfeld module $\phi$, so that $-x^{q^m}+gx^{q^j}+x=0$, implying
$$g=\dfrac{x^{q^m}-x}{x^{q^j}}.$$
Then it was shown in \cite{BBGS15} that $\lambda=\tau^k-x^{q^k-1}$ is an isogeny from $\phi$ to $\psi$, where $\psi=-\tau^m+h\tau^j+1$ (i.e., $\lambda \cdot \phi_T=\psi_T \cdot \lambda$), where
$$h=\dfrac{x^{q^m}-x}{x^{q^m-q^k+1}}.$$

Hence equation \eqref{eq:BBGStower} can be interpreted as describing a relation between $T$-torsion points of isogenous Drinfeld modules in $\mathfrak D_{m,j}$. Note that the supersingular Drinfeld module $-\tau^m+1$ gives rise to the splitting locus of the tower $\mathcal H$, just like the supersingular elliptic curves give rise to the many rational points on the classical modular curves $X_0(N)$. More is said in \cite{BBGS15} about equation \eqref{eq:BBGStower} and related equations. In \cite{CCH21}, the key-ingredients from \cite{BBGS15,ABB17} were worked out further and used to give a precise modular description for each of the function fields in tower $\mathcal H$. Considering Drinfeld modules as in $\mathfrak D_{m,j}$, but in much greater generality, was used to obtain curves with many rational points over $\F{q^m}$ in \cite{Gekeler19}. There the genus computation was carried out using very different techniques, more reminiscent to those in \cite{Gekeler86}.

One of the points of this subsection was to show that modular theory helps if one wants to find equations for good recursive towers. It is quite natural to extend Elkies's phantasia that perhaps any recursive optimal tower has a modular interpretations, to recursive towers over non-square, non-prime finite fields, provided the limit of the tower is the right-hand side of equation \eqref{eq:non-square bound}. In particular, the tower over cubic finite fields from \cite{BezGS05} and \cite{BaGS08} then need a modular interpretation. A modular interpretation for various good towers over cubic finite fields was subsequently given in \cite{ABB17} using Drinfeld modules of rank three.

\section{Possible future directions}

In this final section, we would like to outline some possible future directions. Of course in the previous section already several questions and open problems were raised, which already indicate possible research topics in this area, but here we will be a bit more speculative.

\subsection{Good towers and invariant differentials}

As we have seen, most optimal recursive towers have a modular interpretation and also several recursive towers over $\F{q^3}$ with limit (at least) $2(q^2-1)/(q+2)$ have a modular interpretation. In this section, we would like to highlight several good recursive towers and show a similarity they have with modular towers.

First of all, as mentioned previously some research has been done in determining properties a good recursive tower must have \cite{GS00,BGS04,BGS05,BGS06}. This type of research has shown useful to easy the search for new examples of good recursive towers and likely this type of research can be expanded upon. It has been speculated that a recursive good tower must have a splitting place $P$ of $F_i$ for some $i$, and that a recursive good tower must have a finite ramification locus. Ihara captured in \cite{Ihara07} both at the same time for a good tower over $\F{q^3}$ in terms of an \emph{invariant differential}. He earlier considered such differentials for modular curves in \cite{Ihara71}. Let us explain what this concept means starting with an example, which was communicated to us by Irene Bouw.
\begin{example}
Consider the tower satisfying the recursion from equation \eqref{eq:BBGStower}. Then considering differentials, we see that $-x_i^{-q^j}d x_i = x_{i-1}^{q^k-2}d x_{i-1}$. This and equation \eqref{eq:BBGStower} immediately imply that
$$
\omega:=\dfrac{\left(x_i^{q^m}-x_i\right)^{q^j+q^k-2}}{x_i^{(q^m-1)q^j}}(dx_i)^{q^m-q^j-q^k+1}=\dfrac{\left(x_{i-1}^{q^m}-x_{i-1}\right)^{q^j+q^k-2}}{x_{i-1}^{(q^m-1)q^j}}(dx_{i-1})^{q^m-q^j-q^k+1}. $$
The terminology `invariant' is therefore justified, since $\omega$ does not depend on $i$. Now let us consider $\omega$ in the rational function field $F_0=\F{q^m(x_0)}$, denoting by $P_\alpha$ the zero of $x_0-\alpha$ and $P_\infty$ the pole of $x_0$.
Then $$(\omega)_{F_0}=(q^j+q^k-2)\sum_{\alpha \in \F{q^m}\setminus\{0\}} P_\alpha-((q^m-2)q^j-q^k+2)P_0-((q^m-2)q^k-q^j+2)P_\infty.$$
Now when computing the divisor of $\omega$ in $F_n$, all we need to do is to take the conorm of the divisor we just computed and add $q^m-q^j-q^k+1$ times the different divisor. Since we do not know the precise ramification behavior, the different divisor is not known, but using equations \eqref{eq:boundednessb0} and \eqref{eq:boundednessbinf}, we can deduce that
\begin{multline*}
(\omega)_{F_n} \le (q^j+q^k-2)\sum_{\alpha\in \F{q^m}\setminus\{0\}} \sum_{P|P_\alpha} P\\ -((q^m-2)q^j-q^k+2)\sum_{P|P_0}P-((q^m-2)q^k-q^j+2)\sum_{P| P_\infty}P.
\end{multline*}
Now taking degrees, we can deduce that $(q^m-q^j-q^k+1)(2g(F_n)-2) \le N_1(F_n)(q^j+q^k-2)$ and hence, using $j+k=m$, $$\dfrac{N_1(F_n)}{g(F_n)-1} \ge \dfrac{2(q^j-1)(q^k-1)}{(q^j-1)+(q^k-1)}=2\left(\frac{1}{q^j-1}+\frac{1}{q^k-1}\right)^{-1}.$$
\end{example}

It is quite striking that in the above example, the boundedness of the ramification described in equations \eqref{eq:boundednessb0} and \eqref{eq:boundednessbinf} is precisely what ensures that for each $n$ the places of $F_n$ lying above the ramification locus, are precisely the poles of $\omega$. It is not hard to show that the tower $\mathcal D$ recursively defined by equation \eqref{eq:FR} has an invariant differential, namely $(x_i+x_i^2+x_i^4)^{-1}dx_i=(x_{i-1}+x_{i-1}^2+x_{i-1}^4)^{-1}dx_{i-1}$, but for large enough $n$, it can have zeroes lying above the ramification locus. Ihara mentions in \cite{Ihara07} that the existence of certain invariant differentials can be deduced from modular theory, but turns out that many, and possibly all, known recursive good towers have an invariant differential. This may mean that, as we suspect, such recursive good towers have a deeper interpretation. Let us give another example of a recursive tower taken from \cite{BR20}. The tower is interesting on its own right, since it is the only known example of a recursive, good tower over a prime field.
\begin{example}
Let $p \ge 5$ be a prime number. In case $p>5$, let $-n,-b \in \F{p}^*$ be two non-squares satisfying $n \neq \pm b$. Then the tower over $\F{p}$ recursively defined by the equation
$$\frac{x_i^{p+1}+b}{x_i^p-x_i}=\frac{2b(x_{i-1}^{p+1}+n)}{(b+n)(x_{i-1}^p-x_{i-1})}$$
is asymptotically good and has limit $2/(p-2), see \cite{BR20}.$
Also for $p=5$ a good, recursive tower exists with limit $2/(p-2)$. Concretely from \cite{BR20}, one may use the equation  $$\frac{x_i^6+x_i+2}{x_i^5-x_i}=\frac{x_{i-1}^6+x_{i-1}^5+2x_{i-1}+3}{x_{i-1}^5-x_{i-1}}.$$
After some direct calculations, one then obtains that $$\eta:=\frac{(x_i^p-x_i)^{p-1}}{x_i^{p^2}-x_i}dx_i=\frac{(x_{i-1}^p-x_{i-1})^{p-1}}{x_{i-1}^{p^2}-x_{i-1}}dx_{i-1}$$
both for $p=5$ and for $p>5$.

It is again interesting to note that the poles and zeroes of this differential encode information about the splitting and ramified places of the tower. At level $n$ of the tower, $\eta$ has only zeroes above the splitting locus, which is $\F{p} \cup \{\infty\}$, each with valuation $p-2$. Hence $2g(F_n)-2\le (p-2)N_1(F_n)$ implying $$N_1(F_n)/(g(F_n)-1) \ge 2/(p-2).$$
\end{example}

The main question is then: do all recursive, good towers $\mathcal F=(F_n)_{n \ge 0}$ necessarily have an invariant differential $\omega$ whose zeroes encode information about the number of rational places in the tower, and whose poles encode information about the ramification in the tower?

\subsection{Good towers and class field theory}

The second example from the previous subsection shows that recursively defined towers may well have a role to play in the prime field case. However, there is no doubt that class field theory has produced better results in the past. The recursive tower from \cite{BR20} shows that $A(5) \ge 2/3$, a result obtained using class field theory already in \cite{NX98} and improved upon later as mentioned in Subsection \ref{subsec:classfield}.

The question is if a closer relation between class field towers and recursively defined towers is possible. Certain towers of Shimura curves can be described recursively as demonstrated on pages 8--9 in \cite{E98}. These recursive towers have the interesting property that after the first few steps, they become unramified. Hence they are dominated by certain class field towers. One of the examples in \cite{E98} is a tower recursively defined by the equations
$$x_i^3=1-\left(\frac{x_{i-1}+2}{x_{i-1}-1}\right)^3.$$
When reduced modulo two, these equations define an optimal tower over $\F{64}$. Inspired by this example, J.~Wulftange \cite{Wulftange} found a good, recursive tower over $\F{}$ with $\F{}=\F{q^{2(q-1)}}$ if $q \equiv 1 \pmod{4}$ and $\F{}=\F{q^{q-1}}$ otherwise, using the equation $$x_i^{q-1}=1-\left(\frac{x_{i-1}}{x_{i-1}-1}\right)^{q-1}.$$
Some care should be taken: it is not actually known if these equation really do define a tower. The problem is that the equation may become reducible after a few steps. For small $q$ Wulftange showed that the equations remain irreducible and the expectation is that for any $q$ this happens.
If this is true, then the resulting towers are unramified after two steps, so again a subtower of a class field tower is obtained. It would be interesting to know if more subtowers of good class field towers can be defined recursively.

\subsection{Stratifying Ihara's constant after $p$-rank and $a$-number}

So far, we have focused on obtaining lower bounds for Ihara's constant $A(q)$. One can ask more detailed asymptotic information though. One variation involves the $p$-rank. The $p$-rank $\gamma(F)$ of a function field $F$ with constant field $\overline{\F{p}}$, the algebraic closure of the
finite field $\F{p}$, is defined as the dimension over $\F{p}$ of the group of divisor classes of degree
zero of order $1$ or $p$. For a function field defined over a finite field $\F{q}$ , one defines its $p$-rank as the
$p$-rank of the function field obtained by extending the constant field to the algebraic closure of $\F{q}$. It is known that $0 \le \gamma(F) \le g(F).$ If $\gamma(F) = g(F)$, then the function field $F$ is called ordinary. Motivated by applications, see \cite{CCX}, an interest in asymptotically optimal towers with low $p$-rank, arose. Now for a family of function fields $(F_n)_{n \ge 0}$ with full constant fields $\F{q}$, genera tending to infinity, and limit $\lambda=\lim_{n \to \infty} N_1(F_n)/g(F_n)$, so in particular we assume the limit exists, we can define $\delta:=\liminf_{n \to \infty} \gamma(F_n)/g(F_n)$. We will call $\delta$ the asymptotic $p$-rank of the family $(F_n)_{n \ge 0}$. Is it clear that $0 \le \delta \le 1$, since for any function field $F$, $0 \le \gamma(F)/g(F) \le 1$. The question is then, for any given $\gamma$ between $0$ and $1$, what is the maximal value of $\lambda$ under the assumption that $\delta \le \gamma$? We will denote this constant by $A(q,\gamma)$.

Clearly, we have $A(q,1)=A(q)$. As shown in \cite{BB10,CCX}, the first Garcia--Stichtenoth tower $\mathcal E$, which we know is optimal, satisfies $\delta(\mathcal E)=1/(q+1)$. In particular $A(q^2,\delta) = A(q^2)=q-1$ for $\delta \ge 1/(q+1)$. It is currently not known what the smallest value of $\delta$ is such that $A(q,\delta)=A(q)$, not even when $q$ is a square. In \cite[Cor.20]{BB10}, the asymptotic $p$-rank of a good recursive tower $\mathcal G$ over $\F{q^3}$, with $q=p^e$ and limit $2(q^2-1)/(q+2)$ is computed to be $\frac{2\binom{p+1}{2}^e-2}{(q-1)(q+2)}$, which implies that $$A\left(q^3,\frac{2\binom{p+1}{2}^e-2}{(q-1)(q+2)}\right) \ge \frac{2(q^2-1)}{q+2} \quad \text{where $q=p^e$.}$$
In \cite{ABN17} the asymptotic $p$-rank was computed for a tower over $\F{p^3}$ meeting Zink's bound. The computation done in that paper shows that $$A\left(p^3,\frac{p^2 + p + 4}{4(p^2 + p + 1)}\right) \ge \frac{2(p^2-1)}{p+2}.$$
Further results for good tower with low $p$-rank have been obtained in \cite{AST21}. The stratification of Ihara's constant after asymptotic $p$-rank may serve as a way to gain an overview of these kinds of recent results on good towers. Our definition of $A(q,\gamma)$ is somewhat ad hoc and more generally, one may ask the following: given a family of function fields $(F_n)_{n \ge 0}$ with full constant fields $\F{q}$, genera tending to infinity, such that the limits $\lambda=\lim_{n \to \infty} N_1(F_n)/g(F_n)$ and $\gamma=\lim_{n \to \infty} \gamma(F_n)/g(F_n)$ are defined, which pairs of values $(\lambda,\gamma)$ are possible?

A further stratification may be obtained by also considering the (asymptotic) $a$-number. Though as far as we know not motivated by applications, this further stratification makes perfect theoretical sense. Given a function field $F$ over a finite field, the $a$-number $a(F)$ is the $\F{p}$-dimension of the kernel of the Cartier operator when acting on the space of holomorphic differentials. It is known that $0 \le a(F) \le g(F)-\gamma(F).$ In particular, it makes sense to ask the following question: given a family of function fields $(F_n)_{n \ge 0}$ with full constant fields $\F{q}$, genera tending to infinity, such that the limits $\lambda=\lim_{n \to \infty} N_1(F_n)/g(F_n)$, $\gamma=\lim_{n \to \infty} \gamma(F_n)/g(F_n)$, and $\alpha=\lim_{n \to \infty} a(F_n)/(g(F_n)-\gamma(F_n))$ are defined, which triples of values $(\lambda,\gamma,\alpha)$ are possible? Now various further stratifications of Ihara's constant are possible depending on one's point of view.
Here is one possibility: for a given family of function fields $(F_n)_{n \ge 0}$, with full constant fields $\F{q}$, genera tending to infinity, and limits $\lambda=\lim_{n \to \infty} N_1(F_n)/g(F_n)$ and $\gamma=\lim_{n \to \infty} \gamma(F_n)/g(F_n)$, we define the asymptotic $a$-number of the tower as $\beta:=\limsup_{n \to \infty} a(F_n)/(g(F_n)-\gamma(F_n))$. Then we can define $A(q,\gamma,\alpha)$ as the largest value of $\lambda$ for which there exists a family $(F_n)_{n \ge 0}$ with limit $\lambda$, asymptotic $p$-rank $\gamma$ and $\beta \ge \alpha$. To the best of our knowledge, nothing is known about this further stratification and the asymptotic $a$-number has not been calculated for any recursive good tower. We believe this way of stratifying Ihara's constant into a spectrum of constants is a fruitful way of investigating asymptotically good families of curves or function fields further.

\section*{Acknowledgments}

The author would like to thank the organisers of the Benasque conference on Curves over Finite Fields which was held May 25--29 2021, for the invitation to write this overview article. The author would also like to acknowledge the support from The Danish Council for Independent Research (DFF-FNU) for the project \emph{Correcting on a Curve}, Grant No.~8021-00030B.

\end{document}